\documentclass[sigconf]{acmart}
\usepackage{amsmath, amsthm, color}
\usepackage{graphicx}
\usepackage{caption}
\usepackage{mathtools}
\usepackage{enumitem}
\usepackage{xfrac}
\usepackage{verbatim}
\usepackage{tikz}
\usetikzlibrary{matrix}
\usetikzlibrary{arrows}
\usepackage{algorithm}
\usepackage[noend]{algpseudocode}
\usepackage{caption}
\usepackage{subcaption}

\usepackage{makecell}
\usepackage{array}

\newtheorem{theorem}{Theorem}
\newtheorem{proposition}[theorem]{Proposition}
\newtheorem{lemma}[theorem]{Lemma}
\newtheorem{corollary}[theorem]{Corollary}

\theoremstyle{definition}

\newtheorem{remark}[theorem]{Remark}
\newtheorem{conjecture}[theorem]{Conjecture}

\newtheorem{example}[theorem]{Example}

\newcommand{\PP}{\mathbb{P}}
\newcommand{\RR}{\mathbb{R}}

\newcommand{\CC}{\mathbb{C}}
\newcommand{\ZZ}{\mathbb{Z}}
\newcommand{\FF}{\mathbb{F}}

\newcommand{\CI}{\!\perp\!\!\!\perp\!}

\newcommand{\ASigma}{\mathcal{A}_\Sigma}

\copyrightyear{2022}
\acmYear{2022}
\setcopyright{rightsretained}
\acmConference[ISSAC '22]{Proceedings of the 2022 International Symposium on Symbolic and Algebraic Computation}{July 4--7, 2022}{Villeneuve-d'Ascq, France}
\acmBooktitle{Proceedings of the 2022 International Symposium on Symbolic and Algebraic Computation (ISSAC '22), July 4--7, 2022, Villeneuve-d'Ascq, France}
\acmDOI{10.1145/3476446.3536193}
\acmISBN{978-1-4503-8688-3/22/07}

\begin{document}

\fancyhead{}

\title{\bf Marginal Independence Models}

\author{Tobias Boege}
\email{tobias.boege@mis.mpg.de}
\affiliation{%
  \institution{Max-Planck Institute for Mathematics in the Sciences}
  \streetaddress{Inselstra{\ss}e 22}
  \city{Leipzig}
  \country{Germany}
  \postcode{04103}
}

\author{Sonja Petrovi{\'c}}
\email{sonja.petrovic@iit.edu}
\affiliation{%
  \institution{Illinois Institute of Technology}
  \streetaddress{10 West 35th Street}
  \city{Chicago}
  \country{USA}
}

\author{Bernd Sturmfels}
\email{bernd@mis.mpg.de}
\affiliation{%
  \institution{Max-Planck Institute for Mathematics in the Sciences}
  \streetaddress{Inselstra{\ss}e 22}
  \city{Leipzig}
  \country{Germany}
  \postcode{04103}
}

\begin{abstract}
We impose rank one constraints on 
marginalizations of a tensor, given by a simplicial complex.
Following work of Kirkup and Sullivant, 
such marginal independence models can be made toric by a linear change of coordinates.
We study their toric ideals, with emphasis on
random graph models and 
independent set polytopes of matroids.
We develop the numerical algebra of
parameter estimation, using both Euclidean distance and maximum likelihood,
and we present a comprehensive database of small models.
\end{abstract}



\maketitle

\section{Introduction}

This article concerns a class of statistical models
for discrete random variables that was introduced by Kirkup in \cite{Kir}
and studied further by Sullivant in \cite[Section~4.3.2]{Sul}.
Let $X_1,X_2,\ldots,X_n$ be random variables, where
$X_i$ has the finite state space $[d_i] = \{1,2,\ldots,d_i \}$.
A joint distribution for these random variables is a tensor
$P = (p_{i_1 i_2 \cdots i_n})$ of format
$d_1 \times d_2 \times \cdots  \times d_n$ whose
entries are nonnegative real numbers that sum to $1$.
These distributions are elements in the probability simplex $\Delta_{D-1}$, where 
$D = d_1 d_2 \cdots d_n$.
 As is customary in algebraic statistics \cite{AHT, BC, ST, Sul},
we identify this simplex with the projective space $ \PP^{D-1}$.
Furthermore, by a ``model'' we will mean either a subvariety of $\PP^{D-1}$ or its
intersection with $\Delta_{D-1}$, depending on
context. The latter is a semialgebraic set.

For any subset $\sigma$ of $[n] = \{1,\ldots,n\}$,
we write $P_\sigma$ for the corresponding {\em marginalization}.
Thus $P_\sigma$ is a tensor with $\prod_{i \in \sigma} d_i$ entries. These are obtained from
the entries $p_{i_1 i_2 \cdots i_n}$ of $P$ by
summing out the indices not in $\sigma$.
For instance, if $n=5$ and $\sigma = \{2,3\}$ then
$P_\sigma$ is the $d_2 \times d_3$ matrix whose entry in row $j$ and column $k$ equals
$\,p_{+ jk + + } = \sum_{i=1}^{d_1} \sum_{l = 1}^{d_4} \sum_{m=1}^{d_5} p_{ijklm}$.

Let $\Sigma$ be any simplicial complex with vertex set $[n] $. We assume
 $\{i\} \in \Sigma$ for all $i \in [n]$.
The {\em marginal independence model} $\mathcal{M}_\Sigma$ 
is the set of distributions $P \in \Delta_{D-1}$ such that $P_\sigma$ has rank $1$ for all $\sigma \in \Sigma$.
The random variables $\{X_i: i \in \sigma\}$ are completely independent for $\sigma \in \Sigma$.
Since the rank $1$ constraint is the vanishing of the $2 \times 2$ minors of all
flattenings of $P_\sigma$, we see that $\mathcal{M}_\Sigma$ is a 
variety defined by quadratic equations in the tensor space $\PP^{D-1}$.
The following model,
pairwise independence for three variables,
appears in \cite[Section~4.3.2]{Sul}.

\begin{example}[$3$-cycle] \label{ex:3cycle}
Let $n=3$ and
$\Sigma = \bigl\{ \{1,2\}, \allowbreak \{1,3\}, \allowbreak \{2,3\}, \allowbreak
\{1\}, \allowbreak \{2\}, \allowbreak \{3\}, \allowbreak \emptyset \bigr\}$.
The model $\mathcal{M}_\Sigma$ comprises  $d_1 \times d_2 \times d_3 $ tensors
whose three marginalizations are matrices of rank $1$.
This is an irreducible variety of dimension $\prod_{j=1}^3(d_j - 1) + \sum_{j=1}^3 (d_j - 1)$
 in the tensor space~$\PP^{D-1}$. 
 
 The natural constraints
 that describe this variety
 are the various $2 \times 2$ determinants
  \begin{equation}
 \label{eq:notenough}
 \begin{gathered}
 p_{ij+}\,p_{kl+} \,-\, p_{il+}\,p_{kj+}\, , \quad
 p_{i+j}\,p_{k+l} \,-\, p_{i+l}\,p_{k+j}\, , \\ {\rm and} \quad
p_{+ij}\,p_{+kl} \,-\, p_{+il}\,p_{+kj}. 
 \end{gathered}
 \end{equation}
 These do not generate the prime ideal of $\mathcal{M}_\Sigma$.
 See \cite[Section~6]{Kir} for extraneous components.

The binary model (all $d_i=2$), which serves as running example in \cite{Kir},
has dimension~$4$ and degree $5$ in $\PP^7$. Its ideal is generated by five quadrics and is Gorenstein.
The ternary model (all $d_i=3$) has dimension $14$ and degree $43$ in $\PP^{26}$. 
For this model, Kirkup concluded on \cite[page 453]{Kir} that  ``a free resolution cannot be computed''.
This is no longer true. Using the toric representation in Theorem \ref{thm:sethmain},
we easily find the
 {\tt Macaulay2} Betti diagram for $\mathcal{M}_\Sigma$:
\begin{small}
\begin{verbatim}
       0  1   2   3    4    5    6    7    8    9  10 11 12
total: 1 55 303 920 2309 4740 6700 6038 3348 1082 228 61  7
    0: 1  .   .   .    .    .    .    .    .    .   .  .  .
    1: . 55 303 711  759  285    .    .    .    .   .  .  .
    2: .  .   . 209 1550 4455 6699 6029 3276  887  30  .  .
    3: .  .   .   .    .    .    1    9   72  195 198 61  7
\end{verbatim}
\end{small}
Note that the ideal is perfect (Cohen-Macaulay), in accordance with \cite[Conjecture 28]{Kir}. \hfill
$ \triangle$
\end{example}

The toric representation is due to Sullivant \cite[Section~4.3.2]{Sul}, who
proves it for $n=3$ and states that it
``will also work for many other marginal independence models''. Our
 Sections~\ref{sec2} and \ref{sec3} develop this in detail.
We show that each model $\mathcal{M}_\Sigma$ is a toric variety
after a linear change of coordinates. This reaffirms  that
``the world is toric'' \cite[Section~8.3]{MS}. 
The new coordinates were called 
{\em M\"obius parameters} by Drton~and Richardson~\cite{DR}. These authors also treat
marginal independence but their set-up is based on bidirected graphs and it differs
from ours. 
Most of the models in \cite{DR} are not in the class we study here; see
Example \ref{ex:DRcomp}.

This paper is organized  as follows. In Section~\ref{sec2} we introduce
{\em M\"obius coordinates} $q_\bullet$, and we show that these have
   desirable properties.
For instance, for $2 \times 2 \times 2$ tensors, they~are
$
q = p_{+++}, \, \allowbreak
q_1 = p_{1++}, \, \allowbreak
q_2 = p_{+1+}, \, \allowbreak
q_3 = p_{++1}, \, \allowbreak
q_{12} = p_{11+}, \, \allowbreak
q_{13} = p_{1+1}, \, \allowbreak
q_{23} = p_{+11}, \, \allowbreak
q_{123} = p_{111}
$.
In these coordinates, the prime ideal for the binary 3-cycle in Example \ref{ex:3cycle} is the toric ideal
\begin{multline}
\label{eq:3cycle222}
\langle\,
q_1 \,q_2 -q \,q_{12},\,
q_1 \,q_3-q \,q_{13},\,
q_2\, q_3-q \,q_{23}, \\
q_2\, q_{13} - q_1 \,q_{23}\,,\,\,
q_3 \,q_{12} -  q_1 \,q_{23}\, 
\rangle  =  {} \\
\mbox{Pfaffians}_4 \begin{small}
\begin{pmatrix}
   0 & \!\!\! \! -q &  q_2 &    0 & q_3 \\
  q  &  0  &\!\!\! -q_2 & q_1 & 0 \\
 -q_2 &  q_2 & 0 & \!\!\!-q_{12} & \!\!\! -q_{23} \\
  0 & \!\!\!\! -q_1 & q_{12} & 0 & q_{13} \\
-q_3 & 0 & q_{23} &\!\!\! -q_{13} & 0 
\end{pmatrix} \! . \end{small}
\end{multline}

In Section~\ref{sec3} we show that every marginal independence model $\mathcal{M}_\Sigma$ is toric in the
M\"obius coordinates $q_\bullet$. In particular, we identify the associated 
polytope $\mathcal{P}_\Sigma$ and model matrix $\mathcal{A}_\Sigma$.
Section~\ref{sec4} is devoted to the minimal generators of the
toric ideals $I_{\mathcal{A}_\Sigma}$. We start from the matroid case,
where the  simplicial complex $\Sigma$ consists of all independent sets. 
We are especially interested in graphic matroids, where 
$\Sigma$ is the set of all subforests in a graph with $n$ edges.
 This leads to 
a new class of random graph models, featured in Example~\ref{ex:randomgraph}.

In Section~\ref{sec5} we turn to the statistical task of parameter estimation.
Given any empirical distribution $U $, we seek to compute a distribution
$\hat P$ in the model $\mathcal{M}_\Sigma$ that best explains the data $U$.
We study this from the perspectives of maximum likelihood (ML) 
and Euclidean distance (ED), with emphasis on exact solutions to 
the corresponding optimization problems.

In Section~\ref{sec6} we present a complete list of all small marginal independence models with binary states.
This includes all simplicial complexes for  $n \leq 6$ and all matroids for $n \leq 7$.

\section{M\"obius Coordinates and the Segre Variety}
\label{sec2}

Let $\RR^D$ be the space of
real tensors of format $d_1 \times \cdots \times d_n$
and $\PP^{D-1}$ the corresponding projective space.
Following \cite[Section~4.3.2]{Sul},
we define an $\RR$-linear map $\,\varphi: \RR^D \rightarrow \RR^D,
\, P \mapsto Q$.
The entries $q_{\ell_1 \ell_2 \cdots \ell_n}$ of $Q$ are called {\em M\"obius coordinates}.
Here, $\ell_j \in \{1,2,\ldots,d_{j}-1, + \}$.
The M\"obius coordinate
$q_{\ell_1 \ell_2 \cdots \ell_n}$ is simply the
linear  form  $p_{\ell_1 \ell_2 \cdots \ell_n}$ in the 
probability coordinates. That linear form
 is an entry of the marginal tensor $P_\sigma$
where $\sigma = \{j : \ell_j \not= + \}$.

\begin{example}[$n=3$, all $d_i = 3$]
The tensor $Q$ has entries indexed by $\{1,2,+\}^3$, for instance
\begin{align*}
q_{121} &= p_{121}, \\
q_{12+} &= p_{121} + p_{122} + p_{123}, \\
q_{+1+} &=
p_{111} + p_{112} + p_{113} + p_{211} + {} \\
&\hphantom{{}={}} p_{212} + p_{213} + p_{311} + p_{312} + p_{313}.
\end{align*}
The M\"obius coordinate $q_{+++}$ is the sum of all $27$ probability coordinates $p_{ijk}$.
The map $\varphi$
is represented  by an upper-triangular $27 \times 27$ matrix 
with entries in $\{0,1\}$. All $27$ diagonal entries are $1$. Precisely
 $98$ of the $351$ entries above the diagonal are $1$. The  ordering of the basis which gives the upper-triangular form can be found under the ``ternary 3-cycle" example on this project's MathRepo page referred to below. 
\hfill $ \triangle$
\end{example}

The map $\varphi$ is invertible because it is represented by an upper triangular matrix with all diagonal entries equal to one.
Moreover, all entries in the inverse matrix are $+1$, $-1$ or $0$. 
 This is precisely the process of {\em M\"obius inversion} for a poset on $D$ elements, which is the direct product 
 of the posets $\{1,2,\ldots,d_j-1,+\}$, where $+$ is below the other $d_j-1$ elements, which are incomparable.

A nonzero tensor $P$  in $\RR^D$ is said to have {\em rank one} if its entries factorize as follows:
\begin{multline} 
\label{eq:tensorfac}
\quad p_{i_1 i_2 \cdots i_n} \,\, \, = \,\,\,
\lambda \theta^{(1)}_{i_1} \theta^{(2)}_{i_2} \cdots \,\theta^{(n)}_{i_n} \\
\hbox{for} \,\, i_1 \in [d_1], \, i_2 \in [d_2], \ldots, i_n \in [d_n]. 
\end{multline}
Here $\lambda$ is an additional homogenization parameter, and the
other parameters satisfy
\begin{equation}
\label{eq:thetarelation}
\theta^{(j)}_{d_j} \,\, = \,\, 1 - \sum_{k=1}^{d_j-1} \theta^{(j)}_k
\qquad \hbox{for} \,\, j \in [n].
\end{equation}
Let $\mathcal{S}$ denote the set of rank one tensors in $\RR^D$.
The~image of $\mathcal{S}$ in $\PP^{D-1}$ is isomorphic to
$\PP^{d_1-1} {\times}\, \PP^{d_2-1} \times \cdots \times \PP^{d_n-1}$.
The~total number of free parameters in the
parametrization (\ref{eq:tensorfac})~is 
\begin{equation}
\label{eq:paracount} {\rm dim}(\mathcal{S}) \,\, = \,\, d_1 + d_2 + \cdots + d_n - n + 1 . 
\end{equation}
In what follows, we refer
to the affine cone $\mathcal{S}$ in the tensor space $\RR^D$ as the {\em Segre variety}.

The homogeneous prime ideal  $I_\mathcal{S}$ of the Segre variety is a toric ideal. It is known to be
generated by quadratic binomials. Namely, these ideal generators are
the $2 \times 2$ minors of all the flattenings of the tensor $P$.
For more information see  \cite[Sections~8.2 and 9.2]{MS}.
The parametrization $\varphi$ corresponds dually to the map of
polynomial rings $\varphi^*: \RR[Q] \rightarrow \RR[P]$, where
each unknown $q_{\ell_1 \ell_2 \cdots \ell_n}$ is mapped to the
corresponding linear form  $p_{\ell_1 \ell_2 \cdots \ell_n}$.
We further identify the two index sets $\bigtimes_{j=1}^n [d_j]$
and $\bigtimes_{j=1}^n ([d_j-1] \cup \{+\}) $ by simply sending $d_j$ to $+$.
We write $J_\mathcal{S}$ for the toric ideal in $\RR[Q]$ that is obtained
from the toric ideal $I_\mathcal{S}$ in $\RR[P]$ from the resulting identification of labels.
With this notation, we have the following key lemma.

\begin{lemma} \label{lem:key}
The inverse image of the toric ideal $I_\mathcal{S}$ in $\RR[P]$ under the ring map $\varphi^*$ is
the toric ideal $J_\mathcal{S}$ in $\RR[Q]$. In other words, the
linear change of coordinates $\varphi$ preserves the Segre variety $\mathcal{S}$.
In particular, $\mathcal{S}$ is still a toric variety, even when  written in M\"obius coordinates.
\end{lemma}

\begin{example}[$n=2,d_1=d_2=2$]
Consider the most simple case, namely $2 \times 2$ matrices.
The objects of interest are the principal ideals $I_\mathcal{S} = \langle p_{11} p_{22} - p_{12} p_{21} \rangle $ and
$J_\mathcal{S} = \langle q_{11} q_{++} - q_{1+} q_{+1} \rangle$.
The image of $J_\mathcal{S}$ under the ring map $\varphi^*: \RR[Q] \rightarrow \RR[P]$ is generated by the quadratic form
\begin{multline*}
   p_{11} p_{++} - p_{1+} p_{+1} \, = {} \\
   p_{11} (p_{11} + p_{12} + p_{21}+ p_{22} ) - (p_{11} + p_{12})(p_{11} + p_{21}) \, = {} \\
   p_{11} p_{22} - p_{12} p_{21}.
\end{multline*}
This identity shows that $\varphi^*(J_\mathcal{S}) = I_\mathcal{S}$, so 
the first assertion  in  Lemma \ref{lem:key} holds.
\hfill$ \triangle$
\end{example}

\begin{proof}[Proof of Lemma \ref{lem:key}]
We consider the one-to-one para\-metri\-zation of $\mathcal{S}$
given in (\ref{eq:tensorfac}) and~(\ref{eq:thetarelation}).  The image of 
$\varphi^*(q_{i_1 i_2 \cdots i_n}) = p_{i_1 i_2 \cdots i_n}$ under the
associated ring map equals $ \lambda \prod_{j=1}^n \theta^{(j)}_{i_j} $.
Here some of the indices $i_j$ take the value ${+}$. This means that we are
summing over all states in $[d_j]$. The hypothesis (\ref{eq:thetarelation}) ensures that
   $\theta^{(j)}_{i_j} = 1$ whenever $i_j = {+}$. Hence, we conclude~that
\begin{equation}
\label{eq:parafac2}
 q_{i_1 i_2 \cdots i_n} \,\,\mapsto \,\,\,\, \lambda  \! \prod_{j: i_j \in [d_j-1]} \!\!\! \theta^{(j)}_{i_j} 
 \end{equation}
under the composition of $\varphi$ with the parametrization (\ref{eq:tensorfac}) of the Segre variety $\mathcal{S}$.
By definition, the kernel of the ring map (\ref{eq:parafac2}), which is an ideal in $\RR[Q]$, is
the inverse image of $I_\mathcal{S}$ under~$\varphi^*$.

The products on the right of (\ref{eq:parafac2}) are monomials in the free parameters, which 
are counted in (\ref{eq:paracount}). The exponent vectors of these monomials
are the vertices of a polytope affinely isomorphic to
$ \Delta_{d_1-1} \times \Delta_{d_2-1} \times \cdots \times \Delta_{d_n-1}$.
Here we are using the identification of labels that was described  prior to Lemma \ref{lem:key}.
Hence the kernel of (\ref{eq:parafac2}) equals the toric ideal $J_\mathcal{S}$.
\end{proof}

Of special interest in statistics is the case of binary random variables,
where $d_1 = \cdots = d_n = 2$. Here each M\"obius coordinate
is indexed by a string $(i_1 , i_2, \ldots , i_n)$ in $\{1,+\}^n$.  Following \cite[Section~3]{DR}, 
we represent this string by a subset of~$[n]$, namely the set of indices
$j$ where $i_j = 1$. Thus, in the binary case, M\"obius coordinates are indexed by subsets of
$[n]$.

\begin{example}[Binary $3$-cycle]\label{ex:3cycleb}
Let $n=3$. 
The eight M\"obius coordinates $q_A$ are indexed by subsets $A \subseteq [3]$,
with the extra convention $q = q_\emptyset$.
The toric ideal $J_\mathcal{S}$ of the Segre variety $\mathcal{S}$ is generated by
nine quadratic binomials $p_\sigma p_\tau - p_{\sigma \cap \tau} p_{\sigma \cup \tau}$ in these unknowns.
  Now consider the  $3$-cycle model $\mathcal{M}_\Sigma$
  in Example \ref{ex:3cycle}. We write the prime ideal of this model in
   M\"obius coordinates, and we find that it is
generated  by the $4 \times 4$ subpfaffians of the skew-symmetric $5 \times 5$ matrix
shown in \eqref{eq:3cycle222}. This is the toric ideal obtained from $J_\mathcal{S}$ by
eliminating  $q_{123}$.
\hfill $\triangle$
\end{example}

\begin{example}[Binary bowtie] \label{ex:bowtie}
Let $n=5$. The $32$ M\"obius coordinates $q_A$ are indexed by subsets $A \subseteq [5]$.
Here $J_\mathcal{S}$ is generated by $285$ quadratic binomials.
Let $\Sigma$ be the graph with edges $12,13,23,14,15,45$.
As a simplicial complex, $\Sigma$ has $12$ faces, namely six edges, five vertices, and the empty set.
We now eliminate all $20$ unknowns $q_B$ with $B \not\in \Sigma$ from the Segre ideal $J_\mathcal{S}$.
The resulting toric ideal  in  $12$ M\"obius coordinates $q_C$ with $ C \in \Sigma$
defines~$\mathcal{M}_\Sigma$.
This ideal has codimension 
$6$
and degree $15$. 
It is minimally generated by $14$ quadrics and~one cubic, namely 
$q_{45} q_{12} q_{13} - q_{14} q_{15} q_{23}$.
This is the smallest marginal independence model whose ideal requires a cubic generator.
In Section~\ref{sec4} we shall see binomials of higher degree.
\hfill $\triangle$
\end{example}

\section{Toric Representation}	
\label{sec3}

Fix positive integers $d_1,d_2,\ldots,d_n$,
the Segre variety $\mathcal{S}$, and a simplicial complex $\Sigma$ on $[n]$.
The model $\mathcal{M}_\Sigma$ consists of all probability tensors $P$ with marginals
 $P_\sigma$ of rank one for all $\sigma \in \Sigma$.
Let $\mathcal{L}_\Sigma$ denote the linear subspace of all tensors $P$
in $\RR^D$ whose marginals $P_\sigma$ are zero for all $\sigma \in \Sigma$.
This subspace and our marginal independence model are related as follows:

\begin{lemma} \label{lem:Mdeco}
A tensor lies in the model $\mathcal{M}_\Sigma$ if and only if it is a
sum of a rank one tensor and a tensor in $\mathcal{L}_\Sigma$. In symbols,
we have the following rational parametrization of our variety:
\begin{equation} \label{eq:Mdeco}
\mathcal{M}_\Sigma \,\, = \,\, \mathcal{S} \,+ \, \mathcal{L}_\Sigma.
\end{equation}
\end{lemma}

\begin{proof}
This is \cite[Theorem 1]{Kir}, and the argument is as follows.
Consider the $1$-marginals  $P_{\{1\}} \in \RR^{d_1}, \ldots, P_{\{n\}} \in \RR^{d_n}$.
Their tensor product  $P' = P_{\{1\}} \otimes \cdots \otimes  P_{\{n\}}$
is  a rank one tensor  in $\RR^D$. Assuming without loss of generality that
the entries of $P$ sum to $1$, this construction yields also
$P'_{\{i\}} = P_{\{i\}}$ for all $i \in [n]$.
Assume now $P \in \mathcal{M}_\Sigma $, so $P_\sigma$ has rank one for 
$\sigma \in \Sigma$. By independence this implies $P'_\sigma = P_\sigma$ for
all $\sigma \in \Sigma$.  Therefore, the difference $P-P'$ lies in the linear
space $\mathcal{L}_\Sigma$.
\end{proof}

We now connect the decomposition (\ref{eq:Mdeco}) to the M\"obius coordinates
from Section~\ref{sec2}.
The M\"obius coordinate $q_{i_1 i_2 \cdots i_n}$ is said to be  {\em relevant} if the set
$\,\{ \, j\, :\, i_j \in [d_j-1] \, \}\,$ is a face of~$\Sigma$.
Otherwise, $q_{i_1 i_2 \cdots i_n}$  is called {\em irrelevant}.
In the binary case, the relevant M\"obius coordinates are those indexed by the faces
of $\Sigma$, and the irrelevant ones are indexed by the nonfaces.

\begin{lemma} \label{lem:relevant}
The linear space $\mathcal{L}_\Sigma$ is the common zero set of
all relevant M\"obius coordinates  $q_{i_1 i_2 \cdots i_n}$. Hence its dimension equals
the number of irrelevant M\"obius coordinates, which~is
\begin{equation}
\label{eq:serkan} 
{\rm dim}(\mathcal{L}_\Sigma) \,\, = \,\, \sum_{\tau \not\in \Sigma} \prod_{j \in \tau} (d_j-1) . 
\end{equation}
\end{lemma}

\begin{proof}
This result appears in \cite[Theorem 2.6]{HS}, albeit in a slightly different context.
For completeness, we offer a proof.
We identify each  M\"obius coordinate with its image under $\varphi^*$, which is a sum of $p$-coordinates.
With this, each relevant M\"obius coordinate $q_\bullet$ is among the entries of the marginal tensor $P_\sigma$
where $\sigma = \{ \, j\, :\, i_j \in [d_j-1] \, \}$. Since $\sigma \in \Sigma$, the linear form
$q_\bullet$ vanishes on $\mathcal{L}_\Sigma$.
Now, it remains to consider entries $p_\bullet$ of $P_\sigma$ whose indices
involve $d_j$ for some $j \in \sigma$. Each such entry is an alternating sum of
relevant M\"obius coordinates supported on faces contained in $\sigma$. In particular,
$p_\bullet$ vanishes whenever all relevant M\"obius coordinates vanish.
This proves the first assertion. The second assertion is obtained by counting
irrelevant coordinates according to their support $\tau = \{ \, j\, :\, i_j \in [d_j-1] \, \}$.
\end{proof}

\begin{example}
The expression (\ref{eq:serkan}) is a polynomial in $d_1,\ldots,d_n$, which often simplifies greatly.
Consider the bowtie graph in Example \ref{ex:bowtie}.
The sum (\ref{eq:serkan}) over all $20$ nonfaces of $\Sigma$ becomes
\begin{multline*}
d_1 d_2 d_3 d_4 d_5 - d_1 d_2 - d_1 d_3 - d_2 d_3 - d_3 d_4 - {} \\
d_3 d_5 - d_4 d_5 + d_1 + d_2 + 3 d_3 + d_4 + d_5 - 2.
\end{multline*}
Of course, this expression evaluates to $20$ for $d_1= \cdots = d_5= 2$. 
The $20$ coordinates $q_B$ referred to in Example \ref{ex:bowtie}
are irrelevant while the $12$ coordinates $q_C$ are relevant.
\hfill $\triangle$
\end{example}

We now present the Toric Representation Theorem for marginal independence models.

\begin{theorem} \label{thm:sethmain} The variety $\mathcal{M}_\Sigma$ is irreducible, and its prime ideal
is toric in M\"obius coordinates. It is obtained
from the Segre ideal $J_\mathcal{S}$ by eliminating all M\"obius coordinates that are irrelevant.
Viewed modulo the linear space $\mathcal{L}_\Sigma$, the model
$\mathcal{M}_\Sigma$ is the toric variety given by the monomial parametrization (\ref{eq:parafac2})
where $q_{i_1 i_2 \cdots i_n}$ runs over all relevant M\"obius coordinates.
\end{theorem}

\begin{proof}
This was shown for $n=3$ by Sullivant in \cite[Section~4.3.2]{Sul}.
The proof for general $n$ is analogous.
In our set-up, it consists of Lemmas  \ref{lem:key}, \ref{lem:Mdeco} and \ref{lem:relevant}. Equation (\ref{eq:Mdeco}) says that
$\mathcal{M}_\Sigma$ is the cone over a projection of the Segre variety $\mathcal{S}$.
This projection has the center $\mathcal{L}_\Sigma$, when viewed in $\PP^{D-1}$.  Algebraically, this corresponds to
passing to the relevant M\"obius coordinates, by  Lemma~\ref{lem:relevant}.
Lemma \ref{lem:key} furnishes the remaining step. Since the Segre variety $\mathcal{S}$ is toric
in M\"obius coordinates, so is its projection onto the subset of relevant coordinates.
\end{proof}

\begin{corollary} \label{cor:Mdimen}
The marginal independence model has the dimension expected from (\ref{eq:Mdeco}), i.e.
\begin{multline}
\label{eq:dimformula}
{\rm dim}(\mathcal{M}_\Sigma) \,\, = 
{\rm dim}(\mathcal{S}) \,+\,
{\rm dim}(\mathcal{L}_\Sigma) \,\, = {} \\
\,\, \sum_{i=1}^n (d_i -1) \,\, + \,\,
 \sum_{\tau \not\in \Sigma} \prod_{j \in \tau} (d_j-1) . 
\end{multline}
In particular, in the binary case, this dimension equals $n$ plus the number of nonfaces of $\,\Sigma$.
\end{corollary}

\begin{proof}
The hypothesis $\{i\} \in \Sigma$ for all $i$ ensures that the projection 
with center $\mathcal{L}_\Sigma$ is birational on the Segre variety $\mathcal{S}$.
Indeed, all model parameters in (\ref{eq:parafac2}) can be recovered rationally from the
relevant M\"obius coordinates: the quadratic monomials  $\lambda \theta^{(i)}_{j}$  occur 
for all $i \in [n]$ and $j \in [d_i-1]$, and the special parameter $\lambda$ occurs for the empty face in $\Sigma$.
\end{proof}

In toric algebra \cite[Chapter 8]{MS}, one represents a toric variety by the integer matrix 
whose columns are the exponent vectors in a monomial parametrization.
In our setting, we denote this matrix by $\mathcal{A}_\Sigma$.
It has $1 + \sum_{i=1}^n(d_i-1)$ rows,  indexed by the
 model parameters $\lambda$ and $\theta^{(i)}_j$.
 Its   columns are
 indexed by the relevant M\"obius coordinates.
 To be precise, the columns of $\mathcal{A}_\Sigma$ are
the $0$-$1$ exponent vectors of the relevant monomials seen on the right in~(\ref{eq:parafac2}).
The convex hull of these vectors is the model polytope $\mathcal{P}_\Sigma = {\rm conv}(\mathcal{A}_\Sigma)$.
This is a full-dimensional subpolytope of 
the product of simplices $\Delta_{d_1-1} \times \cdots \times \Delta_{d_n-1}$ that is
associated with~$\mathcal{S}$.

\begin{example}[Binary $3$-cycle]\label{ex:3cycleMtx}
For Example \ref{ex:3cycle} with $d_i=2$, we obtain the $4 \times 7$ matrix
$$  \mathcal{A}_\Sigma\,\,\, \,=\quad \begin{small} \bordermatrix{  & \,q & q_1 & q_2 & q_3 & q_{12} & q_{13} & q_{23} \cr
\lambda & \,1 &  1 & 1 & 1 & 1 & 1 & 1 \cr
\theta^{(1)}_1 &\, 0 & 1 & 0 & 0 & 1 & 1 & 0 \cr
\theta^{(2)}_1 & \,0 & 0 & 1 & 0 & 1 & 0 & 1 \cr
\theta^{(3)}_1 &\, 0 & 0 & 0 & 1 & 0 & 1 & 1 \cr
}. \end{small}
$$
The  toric ideal $I_{\mathcal{A}_\Sigma}$  is given in
\eqref{eq:3cycle222}. The $3$-dimensional polytope $\mathcal{P}_\Sigma = {\rm conv}(\mathcal{A}_\Sigma)$ arises from
the regular $3$-cube by slicing off one vertex. By contrast, the binary hierarchical model for the $3$-cycle, in the sense
of \cite{HS}, corresponds to a $6$-dimensional polytope with $8$ vertices.
 \hfill $\triangle$
\end{example}

\begin{example}[Binary matroid models] \label{ex:matroid}
Consider any matroid on $[n]$, and let $\Sigma$ be its complex of independent sets.
Then $\mathcal{P}_\Sigma$ is the independent set polytope of the matroid.  \hfill $\triangle$
\end{example}

\section{Ideal Generators}
\label{sec4}

We are interested in the toric ideal $I_{\ASigma} $ of the marginal
independence model $\mathcal{M}_\Sigma$ in  M\"obius coordinates $q_\bullet$.
For binary models arising from matroids (Example \ref{ex:matroid}), we conjecture that
$I_{\ASigma}$ is generated by quadrics. This is closely related to a famous conjecture due to Neil White
\cite[Conjecture 13.16]{MS} which asserts quadratic generation for the
toric ideal of the matroid base polytope. We here consider the polytope
of all independent sets.
 White's conjecture has been proved for many special cases, e.g.~for
graphic matroids by Blasiak \cite{Blasiak}, for rank three matroids by Kashiwabara \cite{Kashi}, and  up to saturation in~\cite{LM}.

 For our binary matroid models,
 the toric ideal $I_{\ASigma}$ is normal (the proof is analogous to the base polytope case; see \cite[Theorem 13.8]{MS}) and hence it is Cohen-Macaulay.
In general, the Cohen-Macaulay property for simplicial complexes was asserted in \cite[Conjecture 28]{Kir}, along with the
expected dimension for $\mathcal{M}_\Sigma$. We proved the second part of Kirkup's conjecture
in Corollary \ref{cor:Mdimen}, but the first part remains open. 
Presently, we know no ideal
$I_{\ASigma}$ that fails to be Cohen-Macaulay.
In particular, we do not yet know how to transfer the counterexamples in \cite{BTO}
to our setting.
This issue is subtle because non-normal models exist (see \cite{MOS} and Remark \ref{rmk:nono}).

Our next result  states that the toric ideals
$I_{\ASigma}$ can have minimal generators of arbitrarily high degree.
This is based on a construction that extends
the bowtie in Example~\ref{ex:bowtie}.

\begin{theorem} 
\label{thm:binaryHighDegree} 
Let $\Sigma$ be the  $1$-dimensional simplicial complex associated with the graph
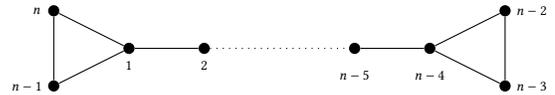
\begin{center}
\begin{tikzpicture}
  [scale=1,auto=center,every node/.style={circle,fill=black,scale=0.5}]
  \node (n1) [label=below:{\Large$1$}] at (0,0)  {};
  \node (n2) [label=below:{\Large$2$}] at (1,0)  {};
  \node (nminus5) [label=below:{\Large$n-5$}]  at (3,0)  {};
  \node (nminus4) [label=below:{\Large$n-4$}]  at (4,0)  {};
  \node (nminus3) [label=right:{\Large$n-3$}]  at (5,-0.5)  {};
  \node (nminus2) [label=right:{\Large$n-2$}]  at (5,0.5)  {};
  \node (nminus1) [label=left:{\Large$n-1$}]  at (-1,-0.5)  {};
  \node (n) [label=left:{\Large$n$}]  at (-1,0.5)  {};
  \foreach \from/\to/\weight in {n1/n2/1,  nminus5/nminus4/1, nminus4/nminus3/1, nminus3/nminus2/1,nminus4/nminus2/1, nminus1/n1/1, nminus1/n/1,n/n1/1}
    \draw (\from) -- (\to);
  \foreach \from/\to/\weight in {n2/nminus5/1}
    \draw[dotted] (\from) -- (\to);
\end{tikzpicture}
\end{center}
Then the toric ideal $I_{\ASigma}$ for the binary model $\mathcal{M}_\Sigma$ has a minimal generator of degree $n-2$.
\end{theorem}

\begin{proof} 
Consider the following binomial in the M\"obius coordinates associated with the edges:
\begin{align*}
\mbox{[$n$ even]} \quad f &= q_{n-1,n}  q_{12}^2q_{34}^2\cdots q_{n-5,n-4}^2  q_{n-3,n-2} - {} \\
&\hphantom{{}={}} q_{1,n-1}q_{1n} q_{23}^2q_{45}^2\dots q_{n-6,n-5}^2 q_{n-4,n-3}q_{n-4,n-2}, \\
\mbox{[$n$ odd]} \quad  f &= q_{n-1,n}     q_{12}^2q_{34}^2\cdots q_{n-6,n-5}^2  q_{n-4,n-3}q_{n-4,n-2} - {} \\
&\hphantom{{}={}} q_{1,n-1}q_{1n}    q_{23}^2q_{45}^2\dots q_{n-5,n-4}^2   q_{n-3,n-2}.
\end{align*}
The binomial $f$ has degree $n-2$. Each of its two monomials
uses the indices $1,2,\ldots,n-4$ twice, and it uses the indices $n-3,n-2,n-1,n$ once.
Hence $f$ vanishes under the
monomial map   in (\ref{eq:parafac2}), which sends  $q_{ij}$ to $ \lambda \theta^{(i)}_1 \theta^{(j)}_1$,
and this means that $f$ lies in the toric ideal  $I_{\ASigma}$.
 We claim that $f$ occurs in every minimal generating set of $I_{\ASigma}$.
In the language of algebraic statistics \cite{AHT}, $f$ is {\em indispensable} for the Markov basis of $\ASigma$, i.e.~the
corresponding fiber of the nonnegative integer linear map 
given by $\mathcal{A}_\Sigma$ has only two elements.

This can be shown by means of  a hypergraph technique. 
Namely, since  $\ASigma$ is $0/1$, it is an incidence matrix of  \emph{the parameter hypergraph}  $\mathcal H_\Sigma$  of the model. Defined in \cite[Section 3.1]{PS}, vertices of $\mathcal H_\Sigma$ are $\{\lambda \} \cup \{\theta^{(i)}_1\}_{i=1}^n$, and edges are of size $1$, $2$, and $3$, representing monomial images of   $q_\emptyset$, $q_i$, and $q_{ij}$, respectively. 
A  pair of edge multisets  $\mathcal E:=\mathcal E_{blue} \cup \mathcal E_{red}$ in   $\mathcal H_\Sigma$ is said to be \emph{balanced} if $\deg_{\mathcal E_{blue}}(v)=\deg_{\mathcal E_{red}}(v)$ for each vertex $v$. 
Denote by $f_{\mathcal E}:=f_{\mathcal E_{red}}-f_{\mathcal E_{blue}}$ a binomial supported on $\mathcal E$; by \cite[Proposition 3.1]{GrPe}, every binomial in  $I_{\ASigma}$ is of this form. 
\cite[Proposition 4.1]{GrPe} states that if there exists no \emph{splitting set} of the balanced edge set $\mathcal E$ supporting the binomial, then the binomial  is indispensable. A splitting set of a $\mathcal E$ is a set $S$ of edges in the hypergraph such that $\mathcal E+S$ can be decomposed into two other balanced sets overlapping on $S$, a condition that in  $I_{\ASigma}$  translates to $f_{\mathcal E}$ being generated by smaller binomials supported on subsets $\Gamma_1$ and $\Gamma_2$. Precisely,   $\mathcal E+S = \Gamma_1\sqcup\Gamma_2$;  $S=\Gamma_{1red}\cap\Gamma_{2blue}$; and the two new sets respect coloring: $\Gamma_{1red},\Gamma_{2red}\subset\mathcal E_{red}\sqcup S$,  and similarly for blue.

For our $f$, a splitting set $S$ must correspond to proper faces of $\Sigma$,  
 and due to the color-balancing requirement, the only choices that could  split $\mathcal E$ are the middle vertices of the graph, 
 that  is,  the size-2 edges $\{\lambda\theta_1^{(i)}\}$ for $1\leq i\leq n-4$ in the hypergraph.  If any hypergraph edge $\{\lambda\theta_1^{(i)}\} \in S$, then  $\Gamma_1$  
   represents the collection to the left 
     of vertex $i$.  Then $\deg_{blue}(\lambda)\geq 1+\deg_{red}(\lambda)$ in $\Gamma_1$,  
 requiring the addition of the singleton $\{\lambda\}$ to $\Gamma_{1blue}$. In  M\"obius coordinates,  this means $q_\emptyset$ is required in the support to make the binomial supported on $\Gamma_1$ homogeneous and thus living in $I_{\ASigma}$. But this addition violates the splitting set definition, and as it is the only way to color-balance, a splitting set does not exist. 
\end{proof}

\begin{remark} \label{rmk:nono}
For $n \geq 7$, both terms of the indispensable binomial $f$ 
contain the square of a coordinate $q_\bullet$.
From this one can infer that $I_{\ASigma}$ is not normal.
We refer to \cite{MOS} for a systematic construction of
simplicial complexes $\Sigma$ whose
 binary model $I_{\ASigma}$ is not normal.
\end{remark}

Here is another example with a high-degree generator which we found quite interesting.

\begin{example}[Sextic for $n=6$] \label{ex:sextic}
Fix the binary model for the  $2$-dimensional complex~with facets
$123$, $146$, $245$, $124$, $156$, $236$, $345$.  The ideal
$I_{\ASigma}$ is minimally generated by $161$ quadrics and one sextic, namely 
${q_{123}^2 q_{146}^2 q_{245}^2-q_{124}^3q_{156} q_{236}q_{345}}$.
All $161$~quadrics are squarefree.
\hfill $\triangle$
\end{example}

We next offer a brief comparison to the graphical set-up of Drton and Richardson in \cite{DR}.

\begin{example} \label{ex:DRcomp}
We fix the bidirected graph in the article \cite{DR} to be the $5$-cycle. This represents the model defined by
$\, \bigl\{
{1 \CI 34},\, \allowbreak
{2 \CI 45},\, \allowbreak
{3 \CI 15},\, \allowbreak
{4 \CI 12},\, \allowbreak
{5 \CI 23} \bigr\}$.
Its ideal is also toric in M\"obius coordinates, by \cite[Theorem 1]{DR}.
 A saturation step reveals that it is  generated by $25$ quadrics:
\begin{multline*}
\langle \,q_{134} q-q_1 q_{34},\,q_{245} q-q_2 q_{45},\,q_{135} q-q_3 q_{15}, \\
  q_{124} q-q_4 q_{12}, \,q_{235} q-q_5 q_{23},\,q_{13} q-q_1 q_3, q_{24}q-q_2 q_4, \\
  q_{35} q-q_3 q_5, \,q_{14} q-q_1 q_4, \,q_{25} q-q_2 q_5\,\rangle\, : \langle q \rangle^\infty.
\end{multline*}
That bidirected $5$-cycle model has dimension $21$ and degree $83$. In the simplex $\Delta_{31}$, 
it is naturally sandwiched between two of our models.
It contains $\mathcal{M}_\Sigma$ for $\Sigma = [134,245,135,124,235]$, which has
dimension $16$, 
degree $68$, and $84$ quadrics.
And, it is contained in the model
$\mathcal{M}_{\Sigma'}$ for $\Sigma' = [13,24,35,14,25]$, which has
dimension  $26$, 
degree $12$, and $10$ quadrics.
\hfill $\triangle $
\end{example}

The original motivation for this project  came from algebraic models for random graphs.
Marginal independence for binary random variables offers a natural class of such models.
Let $n = \binom{s}{2}$ and consider undirected simple
graphs on $s$ vertices, with edges labeled by $[n]$. These graphs are in bijection with subsets of $[n]$.
 Probability distributions on graphs are points in the simplex $\Delta_{2^n-1}$.
Example \ref{ex:matroid} introduces a new model for marginal independence of edges in a graph. This  model is 
  associated with the binary graphic matroid of the complete graph $K_s$, and  $\Sigma $ is the simplicial complex of all forests.

\begin{example}[Random graph model] \label{ex:randomgraph}
We consider graphs on $s=4$ nodes $a,b,c,d$. The~$n=6$
possible edges are labeled $1{:} ab, \allowbreak
2{:}ac, \allowbreak 3{:}ad, \allowbreak 4{:} bc, \allowbreak
5{:} bd, \allowbreak 6{:} cd $.
The associated simplicial complex $\Sigma$ is $2$-dimensional and it has $1+6+15+16=38$ simplices.
These simplices are the $38$ forests
with vertices $a,b,c,d$. The facets of $\Sigma$ are the $16$ spanning trees in the
complete graph $K_4$. Explicitly, they are the
 triples in $[6]$ other than the cycles $124, 135, 236, 456 $.
 Each of the $2^n = 64$ graphs with vertex set  $\{a,b,c,d\}$ has a certain
 probability $p_{i_1 i_2 \cdots i_6}$, according to 
our model. Being in $\mathcal{M}_\Sigma$ means that the edges in any
spanning tree are chosen independently. However, choices of edges are no longer 
independent when a cycle is formed.

The random graph model $\mathcal{M}_\Sigma$ has dimension~$32$
and lives in $\Delta_{63}$.
In the $38$ M\"obius coordinates, it is given by
a toric ideal $I_{\mathcal{A}_\Sigma}$ of codimension $31$ and degree $320$.
Here $\mathcal{P}_\Sigma = {\rm conv}(\mathcal{A}_\Sigma)$
is the independent set polytope of the graphic matroid of $K_4$, which
 has dimension~$6$ and volume $320$.
The ideal $I_{\mathcal{A}_\Sigma}$ is generated by $358$ quadratic binomials, including
\begin{gather*}
q_1 q_2 -q \, q_{12}, \,\, q_{13} q_5-q_1q_{35}, \\
 q_{34} q_{16}-q_{1}q_{346}, \,\,  q_{34} q_{125}-q_{12} q_{345}, \\
 q_{346} q_{256} - q_{246} q_{356}.
\end{gather*}
 Each of these is a quadratic constraint in the probability coordinates
 $p_{i_1 \cdots i_6}$.
 These are the probabilities of the $64$ graphs.$\triangle$ \end{example}

\section{Parameter Estimation}
\label{sec5}

\begin{table*}[h]
\begin{tabular}{rcccccccccc}
     & Deg   & \# Real & $\hat p_{111}$ & $\hat p_{112}$ & $\hat p_{121}$ & $\hat p_{122}$ & $\hat p_{211}$ & $\hat p_{212}$ & $\hat p_{221}$ & $\hat p_{222}$ \\
 aED & $17$  & $1$     & $0.50038$      & $0.25053$      & $0.12563$      & $0.06290$      & $0.03224$      & $0.01614$      & $0.00809$      & $0.00405$      \\
 pED & $6$   & $4$     & $0.49995$      & $0.25003$      & $0.12508$      & $0.06255$      & $0.03161$      & $0.01581$      & $0.00790$      & $0.00395$      \\
 ML  & $1$   & $1$     & $0.49610$      & $0.25096$      & $0.12645$      & $0.06397$      & $0.03307$      & $0.01673$      & $0.00843$      & $0.00426$
\end{tabular}
\caption{The degrees and approximate real solutions to the three estimation
problems for the data in Example~\ref{ex:tobiasdata}.}
\label{tab:tobiasdata}
\end{table*}
 
\begin{figure*}[h]
\begin{center}
\[
\begin{matrix}
f_{\rm{ML}} = 131072  x - 65025, \medskip \\
f_{\rm{pED}} = 51906045358364777223177895936000000 x^6-25949282071385528387934418422988800 x^5-
667203267923068375899220639744 x^4 \\ -10201884458974141010339840 x^3+974466168682810262544 x^2-
439726079196072 x+12131289, \medskip \\
f_{\rm{aED}} = 35921448857987957984672857043619023958379363041280 x^{17}-
112184368601411610971663864087396131463766799810560 x^{16} \\ + 
123023056420734950185604640886224384168207042215936 x^{15}-
35318017343233552889505872276403371258008842534912 x^{14} \\ -
31119632092263938978426745207865986613878169010176 x^{13}+
21061388167151685422887633264505958597233035182080 x^{12} \\ +
4192104300175526169054650369772794981753424969728 x^{11}-
8636877987056563398788782025888916859618833989632 x^{10} \\ +
4765228672547153475258729308728421976809238691840 x^9-
2061854808289365276272409996102579828730342932480 x^8 \\+
671335210647435251072776883161281453372064399360 x^7  -
141014273158845796646692769011540039191619698688 x^6 \\ +
28850396232423871007588638133200796814819309568 x^5-
4546486299487890665409123856114280233201184176 x^4 \\ +29996751771765190786011997856469656605245926 x
^3-100528993612875586639200519747729624258489235 x^2 \\ -11767793432133648786334889808703275655521163
 x \,-1233313433553657451832543422209230995792527.
\end{matrix}
\]
\end{center}
\caption{Minimal polynomials for the exact solution coordinate
$x = \hat p_{111}$ of the ML, pED and aED estimators in
Table~\ref{tab:tobiasdata}.}
\label{fig:minpoly}
\end{figure*}

Estimating model parameters from data is a fundamental task in statistics. In our setting,
the data is a tensor $U$ of format $d_1 \times \cdots \times d_n$ whose entries
$u_{i_1 \cdots i_n}$ are nonnegative integers, indicating the number of times each given joint state was observed.
The empirical distribution $\frac{1}{|U|} U $ is a rational point in $\Delta_{D-1}$.
We seek a distribution $\hat P$ in the model $\mathcal{M}_\Sigma$
that best explains the data, in an appropriate statistical sense. We wish to estimate the model parameters 
$(\hat \lambda, \hat \theta)$ that map to~$\hat P$.

We examine three different paradigms for parameter estimation.
These are
Maximum Likelihood (ML) degree, and affine (aED) and projective (pED) Euclidean
distance degree:
\begin{align*}
\tag{ML}  &{\rm max}\,\,\, \sum u_\bullet {\rm log}(p _\bullet)\,\,\,
\hbox{s.t.} \,\,\, p \in \mathcal{M}_\Sigma. \\
\tag{aED} &{\rm min}\,\,\, \sum (u_\bullet - p_\bullet)^2 \,\,\,
\hbox{s.t.} \,\, p \in \mathcal{M}_\Sigma \,\,{\rm and} \,\sum p_\bullet = 1. \\
\tag{pED} &{\rm min}\,\,\, \sum (u_\bullet - p_\bullet)^2 \,\,\,
\hbox{s.t.} \,\,\, p \in \mathcal{M}_\Sigma.
\end{align*}

All three optimization problems are meaningful for data analysis, and also for algebraic geometry.
Recall that $\mathcal{M}_\Sigma$ is an affine cone in $\RR^D$, 
encoding a projective variety in $\PP^{D-1}$.
Problem (pED) asks for the point  in
that affine cone which is closest to $U$. In the algebraic approach we
compute all critical points in $\CC^D$.
In problem (aED) we add the further constraint that the tensor
entries $p_\bullet$ must sum to $1$. Thus (aED) refers
to the Euclidean distance problem for an affine variety in 
the hyperplane $\{\sum p_\bullet = 1\}$ namely
the Zariski closure of the statistical model
obtained by intersecting the cone $\mathcal{M}_\Sigma$ with the
simplex $\Delta_{D-1}$. The subtle distinction between the
projective case and the affine case is discussed  in \cite[Section~6]{EDdeg}.

Problem (ML) is likelihood inference, which is ubiquitous in statistics.
For a geometric introduction see \cite{LikeGeo}.
The log-likelihood function is invariant, up to an additive constant, under
scaling the tensor $P$, and there
is no need to distinguish between affine and projective.

The intrinsic algebraic complexity of a polynomial
optimization problem is the number of complex critical points,
assuming the data tensor $U$ is generic.
For the problems above,
that number is called the {\em ML degree}, the {\em aED degree} and the {\em pED degree},
as in \cite{EDdeg, LikeGeo}.

\begin{example}[$n=3$] \label{ex:tobiasdata}
We fix three binary random variables.
The Segre variety $\mathcal{S} = \mathcal{M}_{2^{[3]}}$ is the set of
$ 2 \times 2 \times 2$ tensors of rank $1$. This variety has
ML degree $1$ and aED degree $17$. Its pED degree is $6$, by \cite[Example 8.2]{EDdeg}.
The three solutions for the data
$ \,(u_{111}, u_{112}, u_{121}, u_{122}, u_{211}, u_{212}, u_{221}, u_{222} ) =
(2^{-1}, \allowbreak 2^{-2}, \allowbreak 2^{-3}, \allowbreak 2^{-4}, \allowbreak
2^{-5}, \allowbreak 2^{-6}, \allowbreak 2^{-7}, \allowbreak 2^{-7})\,$ are displayed
in Table~\ref{tab:tobiasdata}.
 The last row shows the rational numbers  $\hat p_{ijk} = u_{+jk} u_{i+k} u_{ij+}$.
 The solutions for aED and ML sum to $1$, but for pED it does not. For pED there are three other real critical points.

 The three solutions are close to each other in $\RR^8$. However, for an algebraist they are very different.
 To appreciate this, consider the respective minimal polynomials for the solution coordinate $x = \hat p_{111}$
 in Figure~\ref{fig:minpoly}. All three polynomials are irreducible in $\ZZ[x]$, but their size varies considerably.
 \hfill $\triangle$
\end{example}

We performed this computation for a wide range of marginal
independence models $\mathcal{M}_\Sigma$. In what follows we discuss
our methodology. Our findings are summarized in the next section, along
with pointers to a comprehensive database.

We begin with the Euclidean distance problems (aED) and (pED).
Our varieties are toric and hence rational, given both parametrically, by 
Lemma \ref{lem:Mdeco},  and via an implicit representation, 
by Theorem \ref{thm:sethmain}.
 The former leads to an {\em unconstrained optimization problem} whose
 decision variables are the model parameters $(\lambda,\theta)$.
 Its critical equations are simply those in \cite[equation (2.4)]{EDdeg}.
The latter leads to a {\em constrained optimization problem} whose decision
variables are the tensor entries $p_{i_1 i_2 \cdots i_n}$;  the critical equations are given in \cite[equation (2.1)]{EDdeg}.

We solve the unconstrained problem using the {\tt julia} package
{\tt HomotopyContinuation.jl} due to Breiding and Timme \cite{HomotopyContinuation}.  To illustrate
this package, we show the code for  $n=3$:

\begin{small}
\begin{verbatim}
using HomotopyContinuation
@var q1 q2 q3 q12 q13 q23 q123 z
u000, u001, u010, u011, u100, u101, u110, u111 =
   [ 1//2, 1//4, 1//8, 1//16, 1//32, 1//64, 1//128, 1//128 ]
diffs = [u000 - z*(q123), u001 - z*(q12-q123),
   u010 - z*(q13-q123), u011 - z*(q1-q12-q13+q123),
   u100 - z*(q23-q123), u101 - z*(q2-q12-q23+q123),
   u110 - z*(q3-q13-q23+q123),
   u111 - z*(1-q1-q2+q12-q3+q13+q23-q123)]
dist = sum([d^2 for d in diffs])
\end{verbatim}
\end{small}
This sets up the objective function {\tt dist} in M\"obius parameters.
We~next specify the model:
\begin{small}
\begin{verbatim}
model = [q12=>q1*q2, q13=>q1*q3, q23=>q2*q3, q123=>q1*q2*q3]
dist = subs(dist, model...)  # projective ED
vars = variables(dist)
eqns = differentiate(dist, vars)
R = solve(eqns)
\end{verbatim}
\end{small}
This solves the pED problem in Example \ref{ex:tobiasdata}.
If we now run \begin{small} {\tt C = certify(eqns,R)}\end{small}, then this proves correctness,
in the sense of \cite{BRT}.
By deleting entries of {\tt model}, we can specify other complexes $\Sigma$.
For instance,  for the 3-cycle, delete \begin{small} {\tt q123=>q1*q2*q3}\end{small}.
 For aED we use the line
 \begin{small}
\begin{verbatim}
model = [q12=>q1*q2, q13=>q1*q3, q23=>q2*q3,
         q123=>q1*q2*q3, z=>1]  # affine ED
\end{verbatim}
\end{small}
Results of this computation for all binary models up to $n=5$ are 
presented in Section~\ref{sec6}.
Symbolic computation was used to independently
verify a range of small cases. 

Our experiments suggest the following conjecture for all marginal independence models:

\begin{conjecture}
Given any nonnegative tensor $U$ whose entries sum to $1$, the
aED problem has precisely one real critical point,
namely the point $\hat P \in \mathcal{M}_\Sigma \cap \Delta_{D-1}$ that is closest to $U$.
\end{conjecture}

For maximum likelihood estimation, most statisticians
use local hill-climbing methods, such as Iterative Conditional Fitting \cite[Section~5]{DR}.  
By contrast, we here use the global numerical tools of nonlinear algebra
\cite{HomotopyContinuation, BRT}. The objective function for ML is entered as~follows:
\begin{small}
\begin{verbatim}
p = [z*(q123), z*(q12-q123), z*(q13-q123),
   z*(q1-q12-q13+q123), z*(q23-q123),
   z*(q2-q12-q23+q123), z*(q3-q13-q23+q123),
   z*(1-q1-q2+q12-q3+q13+q23-q123)]
loglike = sum([u[i] * log(p[i]) for i = 1:length(u)])
\end{verbatim}
\end{small}
Following \cite{ST}, we solve the rational function equations
given by the gradient of {\tt loglike}. We avoid the numerator polynomials, which are
much too large. The introduction of \cite{ST} discusses numerical ML by stating that
``a key idea is to refrain from clearing denominators''.

Example \ref{ex:tobiasdata} suggests that the
ED degrees exceed the ML degree. However, this is incorrect.
The Segre variety is an exception. Almost all models $\mathcal{M}_\Sigma$ have a larger
  ML degree.

\begin{example}[$n=5$] Fix the bowtie in Example~\ref{ex:bowtie} and consider
a data tensor $U \in \Delta_{31}$. We
estimate the $26$ model parameters, in order to find
the best fit $ \hat P \in \mathcal{M}_\Sigma$. Solving the critical equations for
Euclidean distance is much faster than for maximum likelihood.
We see this from the degrees, which reveal
the number of paths to be tracked in {\tt HomotopyContinuation.jl}.
The aED degree equals $137$ and the pED degree equals $135$.
We do not yet know the ML degree, but
monodromy loops show that it exceeds one million.
An easier model is the claw graph
$\Sigma = [12,13,14,15]$, whose
ML degree equals $14693$. The
aED degree and the pED degree are only $5$ and~$2$, respectively.
\hfill $\triangle$
\end{example}

\begin{remark}
  Some of our largest ML degrees
  still have a small upward margin of error. Those are indicated in Table~\ref{tab:four}
  with a trailing~``$+$''.
  All the  numbers we report are certified lower bounds, thanks to \cite{BRT},
  just like \cite[Proposition~5]{ST}.
  \end{remark}

\begin{table*}[h]
\begin{tabular}{cccccccc}
dimension & degree & mingens & f-vector           & simplicial complex $\Sigma$ &  aED  &  pED  &   ML    \\
   $15$   &  $1$   & $()$    & $(1,4)_5$          & $[1,2,3,4]$                 &   $1$ &   $1$ &     $1$ \\
   $14$   &  $2$   & $(1)$   & $(1,4,1)_6$        & $[3,4,12]$                  &   $5$ &   $2$ &     $1$ \\
   $13$   &  $3$   & $(3)$   & $(1,4,2)_7$        & $[4,12,13]$                 &   $5$ &   $2$ &     $9$ \\
   $13$   &  $4$   & $(2)$   & $(1,4,2)_7$        & $[14,23]$                   &  $25$ &  $25$ &  $1041$ \\
   $12$   &  $4$   & $(6)$   & $(1,4,3)_8$        & $[12,13,14]$                &   $5$ &   $2$ &   $209$ \\
   $12$   &  $5$   & $(5)$   & $(1,4,3)_8$        & $[12,14,23]$                &  $21$ &   $8$ &  $1081$ \\
   $12$   &  $5$   & $(5)$   & $(1,4,3)_8$        & $[4,12,13,23]$              &  $21$ &  $19$ &    $17$ \\
   $11$   &  $6$   & $(9)$   & $(1,4,4)_9$        & $[13,14,23,24]$             &   $9$ &   $3$ &  $1937+$ \\
   $11$   &  $6$   & $(9)$   & $(1,4,3,1)_9$      & $[4,123]$                   &  $17$ &   $6$ &     $1$ \\
   $11$   &  $7$   & $(8)$   &  $(1,4,4)_9$       & $[12,13,14,23]$             &  $33$ &  $25$ &  $2121$ \\
   $10$   &  $8$   & $(13)$  & $(1,4,4,1)_{10}$   & $[14,123]$                  &  $21$ &   $8$ &   $425$ \\
   $10$   &  $9$   & $(12)$  & $(1,4,5)_{10}$     & $[12,13,14,23,24]$          &  $25$ &  $25$ &  $4480+$ \\
    $9$   & $11$   & $(17)$  & $(1,4,5,1)_{11}$   & $[14,34,123]$               &  $57$ &  $29$ &  $1737+$ \\
    $9$   & $12$   & $(16)$  & $(1,4,6)_{11}$     & $[12,13,14,23,24,34]$       &  $73$ &  $85$ & $11885+$ \\
    $8$   & $12$   & $(24)$  & $(1,4,5,2)_{12}$   & $[123,124]$                 &  $25$ &   $8$ &   $201$ \\
    $8$   & $16$   & $(21)$  & $(1,4,6,1)_{12}$   & $[14,24,34,123]$            & $117$ & $112$ &  $8551+$ \\
    $7$   & $18$   & $(28)$  & $(1,4,6,2)_{13}$   & $[34,123,124]$              &  $89$ &  $90$ &  $2121+$ \\
    $6$   & $20$   & $(36)$  & $(1,4,6,3)_{14}$   & $[123,124,134]$             &  $89$ &  $46$ &   $505$ \\
    $5$   & $23$   & $(44)$  & $(1,4,6,4)_{15}$   & $[123,124,134,234]$         & $169$ &  $93$ &   $561$ \\
    $4$   & $24$   & $(55)$  & $(1,4,6,4,1)_{16}$ & $[1234]$                    &  $73$ &  $24$ &     $1$ \\
\end{tabular}
\caption{All $20$ marginal independence models for four binary random variables. \label{tab:four}}
\end{table*}

\section{Classification of Small Models}
\label{sec6}

Every $n$-way tensor has an associated simplicial complex, namely 
the index sets of all marginalizations that have rank $\leq 1$. Conversely,
every simplicial complex $\Sigma$ is realized by some tensor.
Thus, the space of tensors $\RR^D$ has a natural stratification, with cells
 indexed by all simplicial complexes on $[n]$.
Our model $\mathcal{M}_\Sigma$ is the closure of the cell indexed by $\Sigma$.
The map from simplicial complexes to marginal independence models is inclusion-reversing:
\begin{equation} \label{eq:inclusionreversing}
\Sigma \subset \Sigma' \quad \hbox{if and only if} \quad
\mathcal{M}_{\Sigma} \supset \mathcal{M}_{\Sigma'}.
\end{equation}
The geometry of this stratification is important for Bayesian model selection \cite[Chapter~17]{Sul}.

In this section we present the classification of all small models. The number of 
unlabeled simplicial complexes on $n = 1,2,3,4,5,6,7$ vertices equals
$1,2, 5, 20, 180, 16143, 489996795$.
Exhaustive computations are thus limited to $n \leq 6$.
Further, we restrict ourselves to the binary case $d_1 = \cdots = d_n = 2$.
We denote each complex by its list of facets. For instance,
$\Sigma = [12,13,23]$ is the 3-cycle in Example  \ref{ex:3cycle}.
The complete list for $n=4$ is shown in Table \ref{tab:four}.
The last three columns confirm that the ML degree exceeds the
ED degrees for most models.

The first three columns in Table \ref{tab:four} describe
$\mathcal{M}_\Sigma$ as a projective toric variety: its dimension, its degree,
and the number of minimal generators, here all quadratic.
Similar lists of all models for $n=3,4,5,6$ and samples of code in {\tt Macaulay2} and {\tt julia} 
can be found~at 
$$ \hbox{ \url{https://mathrepo.mis.mpg.de/MarginalIndependence}}. $$
For each model, the repository gives both parametrization 
and implicit representation, in machine-readable form, so a
reader can experiment with these in a computer algebra system.

Concerning the minimal generators of $I_{\mathcal{A}_\Sigma}$, we record
the following result from our data.

\begin{proposition} \label{prop:manyideals}
For $178$ of the $180$ models $\Sigma$ with $n=5$, the
ideal $I_{\mathcal{A}_\Sigma}$ is generated by quadrics. The two exceptions are the bowtie 
and the complex $\Sigma = [123,124,134,145,234, 235]$
whose ideal has degree $73$ and
requires $97$ quadrics and $1$ cubic.
For $14104$ of the $16143$ models with $n=6$, the
ideal $I_{\mathcal{A}_\Sigma}$ is generated by quadrics. 
Of the others,  $1930$ ideals require~cubics,
$104$ require quartics, four require quintics, 
and only one (that in Example \ref{ex:sextic}) requires a~sextic.
\end{proposition}

Here is the ideal with the most non-quadratic generators among those in  Proposition \ref{prop:manyideals}:

\begin{example}[$\RR \PP^2$ with $n{=}6$]
Fix $\Sigma = [123, 124, \allowbreak 135, \allowbreak 146, \allowbreak
156, \allowbreak 236, \allowbreak 245, \allowbreak 256, \allowbreak 345, 346]$.
This is the minimal triangulation of the real projective plane.
Its toric ideal $I_{\mathcal{A}_\Sigma}$  is generated by 
$209$ quadrics, $10$ cubics and $15$ quartics. This model
has codimension $25$ and degree $275$.
\hfill $\triangle$
\end{example}

We do not yet have complete results for the algebraic degrees of the estimation problems.
Computing the ML degree seems to be out of reach for $n=5$. The ED degrees are better
behaved. The data provides the following result. 

\begin{proposition}
Among all $180$ binary models for ${n=5}$, 
the largest aED degree equals $1457$, namely for
$\Sigma = [1234, 125, \allowbreak 135, \allowbreak 145, \allowbreak 235, \allowbreak 245, 345]$;
it has ${\it pED} = 1077$. The largest  pED degree equals $1247$, namely 
for the complex $\Sigma$ given by all $10$ triangles; it has ${\it aED} = 1425$.
\end{proposition}

We now turn to marginal independence models defined by matroids (Example~\ref{ex:matroid}).
This includes our random graph models (Example \ref{ex:randomgraph}).
Matroid models are attractive because they connect
independence in linear algebra and independence in statistics.
For realizable matroids,  the binary model $\mathcal{M}_\Sigma$ 
assigns a probability to each subset of $n$ given vectors in a vector space $V$.
Vectors are chosen independently as long as they are independent in $V$.
However, choices of vectors are no longer independent when linear dependencies occur.

\begin{example}[Fano matroid] \label{ex:fano}
The finite vector space $V = \FF_2^3$ has $n=7$ non-zero vectors.
 The complex $\Sigma$ consists of all $28$ bases and their subsets, and
$\mathcal{M}_\Sigma$ is a model for random subsets of $V \backslash \{0\}$.
The number of relevant M\"obius coordinates is $1+7+21+28=57$.
The ideal $I_{\mathcal{A}_\Sigma}$ 
has codimension $49$ and degree $1207$.
It is generated by 
$868$ quadrics; cf.~\cite{Kashi}.
\hfill $\triangle$
\end{example}

We examined all loopless matroids on $n \leq 7$ elements.
Their number is $305$, up to isomorphism.
Their toric ideals $I_{\mathcal{A}_\Sigma}$ are all generated by quadrics.
The ED and ML degree computations do not become
simpler in the matroid case. In fact, the highest ML degree  in Table~\ref{tab:four} is attained for the uniform matroid $U_{2,4}$
and the highest aED degree for $U_{3,4}$. 

The matroids appear among models in our database, where they
are especially marked. Their ED degrees
are available for all models
up to $n=5$, and the ML degrees for all models up to $n=4$.
For some models of special interest, we also computed beyond those limits.

\begin{example}[$n=6$]  Consider the random graph model in Example \ref{ex:randomgraph}.
Given any sample of $4$-vertex graphs, we compute its best fit 
in the model. Using {\tt HomotopyContinuation.jl}, we
determined that the aED degree equals $1981$, while
the pED degree equals $3512$. 
\hfill $\triangle$
\end{example}

We conclude with several directions for future research.
First, consider the statistical results in \cite{DR}.
It would be worthwhile to identify all singularities of
$\mathcal{M}_\Sigma$, with a view towards
\cite[Corollary 3]{DR}. Also, the  Iterative Conditional 
Fitting algorithm \cite[Section~5]{DR} should be developed
for our models. Naturally, it is desirable to determine whether
multiple local optima inside $\Delta_{D-1}$ can occur, for
the various models, under the three estimation paradigms.

For any given matroid, our model strictly contains the 
model of \emph{weak probabilistic representations} due to
  Mat\'{u}\v{s} \cite{Matus}. This expresses higher conditional and functional
(in)dependences. Among its points are tensors
encoding linear representations of the matroid over certain
finite fields, but not all tensors in the Mat\'{u}\v{s} models arise in
this way. It would be interesting to study these models in our setting.
 Example~\ref{ex:fano} is a point of departure.

One class of models to be examined in depth
is the {\em $r$-way  marginal independence model}. This is
the $\mathcal{M}_\Sigma$
given by the uniform matroid of rank $r$ on $[n]$, so
$\Sigma$ is the collection of all subsets of $[n]$ with at most $r$ elements.
A tensor is in $\mathcal{M}_\Sigma$ if and only if all its
$r$-dimensional marginals have rank $\leq 1$.  What can be said about
parameter estimation for these models?

\bigskip
\medskip

\begin{acks}
We are very grateful to Kemal Rose and Simon Telen for helping us with
the software {\tt HomotopyContinuation.jl}. We thank T\`{a}i H\`{a} and Takayuki Hibi for pointers on toric~rings.
SP is partially supported by the Simons collaboration grant for Mathematicians \#854770.
This project began when SP visited MPI-MiS Leipzig in September 2021.
She is eternally grateful for the hospitality of her hosts and everyone at the institute.
\end{acks}

\end{document}